\documentclass[conference]{IEEEtran}
\IEEEoverridecommandlockouts

\usepackage{amsmath,amsfonts,amsthm,amssymb,mathrsfs,bbm}

\usepackage{algorithm,setspace,soul}
\usepackage{algpseudocode}
\usepackage{array}
\usepackage{graphicx}
\usepackage[caption=false,font=normalsize,labelfont=sf,textfont=sf]{subfig}

\usepackage{nicefrac}

\usepackage{balance}
\usepackage{hyperref}

\algrenewcommand\algorithmicrequire{\textbf{Input:}}
\algrenewcommand\algorithmicensure{\textbf{Output:}}

\def\N				{\mathbb N}
\def\Z				{\mathbb Z}
\def\R				{\mathbb R}
\def\C				{\mathbb C}

\def\Mod			{\mathscr M}
\def\Sum			{\mathsf S}
\def\ind			{\mathbbm 1}

\def\Cont			{\mathscr C}
\def\Lebesgue		{\mathrm L}

\def\PW				{\mathsf{PW}}

\def\e				{\mathrm e}
\def\i				{\mathrm i}
\def\T				{\mathrm T}

\def\bfV			{\mathbf{V}}
\def\bfc			{\mathbf{c}}
\def\bfg			{\mathbf{g}}

\def\bfr			{\mathbf{r}}
\def\bfs			{\mathbf{s}}

\def\Ell			{\mathbb L}
\def\E				{\mathbb E}
\def\A				{\mathcal A}

\def\S				{\mathcal S}

\def\complement		{\mathsf c}

\DeclareMathOperator{\supp}{supp}
\DeclareMathOperator{\sgn}{sgn}

\DeclareMathOperator*{\argmax}{arg\,max}
\DeclareMathOperator*{\argmin}{arg\,min}

\newtheorem{definition}{Definition}

\newtheorem{corollary}{Corollary}
\newtheorem{lemma}{Lemma}

\begin{document}

\title{Orthogonal Matching Pursuit based Reconstruction for Modulo Hysteresis Operators
\thanks{This work was supported by the Deutsche Forschungsgemeinschaft (DFG), project number 530863002.}
}

\author{\IEEEauthorblockN{Matthias Beckmann and J\"{u}rgen Jeschke}
\IEEEauthorblockA{Center for Industrial Mathematics, University of Bremen, Germany\\
research@mbeckmann.de $\bullet$ jjeschke@uni-bremen.de}
}

\maketitle
\thispagestyle{plain}
\pagestyle{plain}

\begin{abstract}
Unlimited sampling provides an acquisition scheme for high dynamic range signals by folding the signal into the dynamic range of the analog-to-digital converter (ADC) using modulo non-linearity prior to sampling to prevent saturation.
Recently, a generalized scheme called modulo hysteresis was introduced to account for hardware non-idealities.
The encoding operator, however, does not guarantee that the output signal is within the dynamic range of the ADC.
To resolve this, we propose a modified modulo hysteresis operator and show identifiability of bandlimited signals from modulo hysteresis samples.
We propose a recovery algorithm based on orthogonal matching pursuit and validate our theoretical results through numerical experiments.
\end{abstract}

\begin{IEEEkeywords}
Unlimited sampling, generalized modulo operator, Shannon sampling theory, orthogonal matching pursuit.
\end{IEEEkeywords}

\section{Introduction}

The well-known Shannon sampling theorem~\cite{Shannon1948} allows for the recovery of a bandlimited function $g$ from discrete samples $\{g(n\T) \mid n \in \Z\}$ with sufficiently small sampling rate $\T > 0$, namely if sampled at or above the Nyquist rate, which is implemented in hardware by analog-to-digital converters~(ADCs).
The reconstruction formula, however, assumes access to perfect samples and any distortion in the data leads to artifacts in the reconstruction.
One particular bottleneck in practice is that ADCs operate at a fixed dynamic range and an input signal exceeding the threshold leads to saturation or clipping and, hence, we suffer a permanent information loss.

To overcome these limitations, the Unlimited Sampling Framework (USF) was introduced in~\cite{Bhandari2017,Bhandari2020,Bhandari2021} along with mathematically backed reconstruction algorithms, where, prior to sampling, the input function is folded into the dynamic range using modulo arithmetic.
Since then several reconstruction approaches have been proposed, e.g.~\cite{Rudresh2018,Romanov2019,Azar2022,Guo2023}, and USF has been extended to various applications like imaging~\cite{Bhandari2020a}, computerized tomography~\cite{Beckmann2022,Beckmann2024} and radar~\cite{Feuillen2023}.
More general encoding operators were considered~in~\cite{Azar2025}.

The above approaches assume an instantaneous fold whenever the input signal $g$ reaches the dynamic threshold, yielding an encoding that operates pointwise, i.e., the output at time $t$ is determined by $g(t)$.
As opposed to this, in~\cite{Florescu2022,Florescu2022a} a generalized modulo encoder is proposed that accounts for non-instantaneous folds due to hardware imperfections.
This yields an encoding operator $\Mod_H$ with memory in the sense that $\Mod_H g(t)$ not only depends on $g(t)$, but on $g(\tau)$ for $\tau \leq t$.
Because of this, $\Mod_H$ is also called {\em modulo hysteresis operator}.

\begin{figure}
\centering
\includegraphics[width=\linewidth]{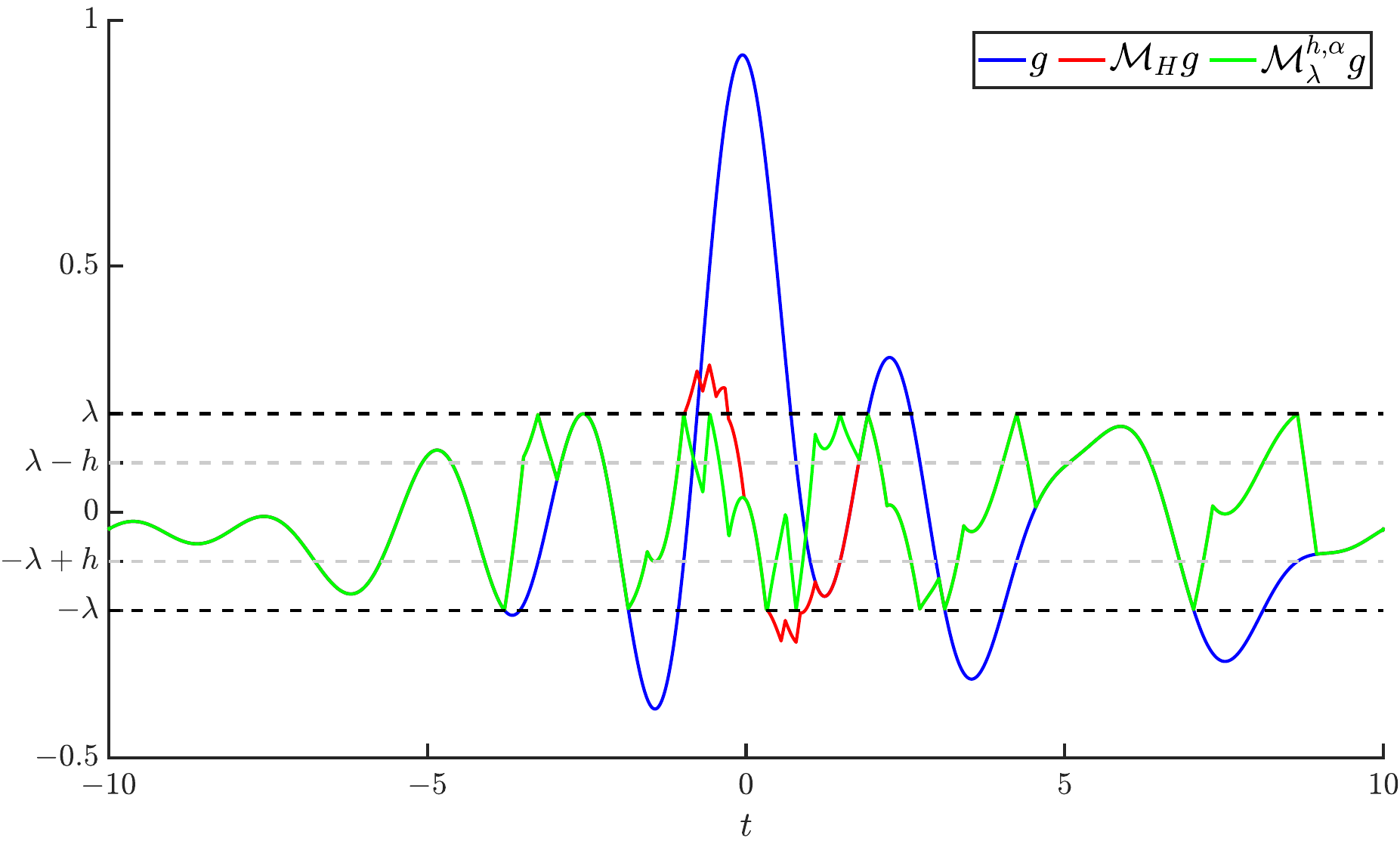}
\caption{Illustration of generalized modulo encoder $\Mod_H$ and modified modulo hysteresis operator $\Mod_\lambda^{h,\alpha}$ for random input $g \in \PW_\Omega$ with $\lambda = 0.2$, $h = 0.1$ and $\alpha = 0.3$, where $\Mod_\lambda^{h,\alpha}g \in [-\lambda,\lambda]$, but $|\Mod_H g(t)| > \lambda$ for some $t$.}
\label{fig:GenMod_example}
\end{figure}

The model proposed in~\cite{Florescu2022,Florescu2022a}, however, does not guarantee that the output is within the dynamic range of the ADC, see Fig.~\ref{fig:GenMod_example} for illustration.
Hence, in this work, we propose a modified modulo hysteresis operator $\Mod_\lambda^{h,\alpha}$ that stays within the given dynamic range.
In Section~\ref{sec:GenMod}, we show that the input signal $g$ is uniquely determined by the discrete samples $\{\Mod_\lambda^{h,\alpha} g(n\T) \mid n \in \Z\}$ with sufficiently small $\T > 0$ and, in Section~\ref{sec:OMP}, explain how $g$ can be recovered using orthogonal matching pursuit (OMP).
Finally, our theoretical results are supported by numerical experiments in Section~\ref{sec:numerics}.

\section{Modulo Hysteresis Operators}\label{sec:GenMod}

For a function $g: \R \to \R$, the ideal modulo encoder $\Mod_\lambda$ with modulo threshold $\lambda > 0$ is pointwisely defined as in~\cite{Bhandari2017} via
\begin{equation*}
\Mod_\lambda g(t) = g(t) - 2\lambda \Bigl\lfloor\frac{g(t) + \lambda}{2\lambda}\Bigr\rfloor
\quad \text{for } t \in \R,
\end{equation*}
where $\lfloor\cdot\rfloor$ denotes the floor function for real numbers, and satisfies $\Mod_\lambda g(t) \in [-\lambda,\lambda]$ for all $t \in \R$.
The typical input space is the Paley-Wiener space $\PW_\Omega$ of square-integrable bandlimited functions with bandwidth $\Omega>0$.
In this case, \cite{Bhandari2020} shows that the input function can be recovered from output samples if the sampling rate satisfies the oversampling condition $\T < \frac{1}{2\Omega\e}$, where the reconstruction is based on iterated forward differences in time domain.
A first hardware validation of a modulo encoder is reported in~\cite{Bhandari2021} together with a Fourier domain recovery approach.

To account for non-ideal hardware implementations, in~\cite{Florescu2022} a generalized modulo encoder $\Mod_{H}$ with parameter triplet $H = (\lambda,h,\alpha)$ is introduced, where the {\em hysteresis} parameter $h \geq 0$ models an imperfect alignment of the reset threshold and the post-reset value and the {\em transient} parameter $\alpha \geq 0$ models an inexact folding transition.
For a function $g$, the output $\Mod_H g$ is defined via a sequence of folding points $(\tau_i)_{i \in \N}$, given by
\begin{gather*}
\tau_1 = \min\bigl\{t > \tau_0 \bigm| \Mod_\lambda (g(t)+ \lambda) = 0\bigr\}, \\
\tau_{i+1} = \min\bigl\{t > \tau_i \bigm| \Mod_\lambda (g(t) - g( \tau_i) + h s_i) = 0 \bigr\}
\end{gather*}
with $s_i = \sgn(g(\tau_i) - g(\tau_{i-1}))$ and $\tau_0 \in \R$.
Then, for $t \in \R$,
\begin{equation*}
\Mod_H g(t) = g(t) - (2\lambda - h) \sum\nolimits_{i \in \N} s_i \, \varepsilon_\alpha(t - \tau_i),
\end{equation*}
where $\varepsilon_0(t) = \ind_{[0,\infty)}(t)$ models an instantaneous folding transition and, for $\alpha > 0$, $\varepsilon_\alpha(t) = \frac{t}{\alpha} \ind_{[0, \alpha)} (t) + \ind_{[\alpha, \infty)}(t)$ models the folding transition as a straight line with slope $\nicefrac{1}{\alpha}$.

Note that the output $\Mod_H g(t)$ depends on all folding points $\tau_j < t$ and, hence, $\Mod_H$ is not acting pointwisely.
Moreover, the above definition does not guarantee that $\Mod_\lambda g(t) \in [-\lambda,\lambda]$ for $t \geq \tau_0$, as illustrated in Fig.~\ref{fig:GenMod_example}.
One reason for this is that the folding points $(\tau_i)_{i \in \N}$ are independent of the transient~$\alpha$.
To account for this, we now wish to propose a modified definition of the modulo hysteresis encoder.
To this end, we rewrite $\Mod_H g$ by defining the function sequence $(\eta_n^{(\alpha)})_{n \in \N_0}$ via
\begin{equation*}
\eta_{n+1}^{(\alpha)} = \eta_n^{(\alpha)} - (2\lambda - h) \sgn(\eta_n^{(\alpha)}(\tau_{n+1})) \, \varepsilon_\alpha(\cdot - \tau_{n+1}),
\end{equation*}
where $\eta_0^{(\alpha)} = g$.
With this, we then obtain
\begin{equation*}
\tau_i = \inf \bigl\{t > \tau_{i-1} \bigm| |\eta_{i-1}^{(0)}(t)| \geq \lambda \bigr\}
\end{equation*}
and
\begin{equation*}
\Mod_H g(t) = \lim_{n \to \infty} \eta_{n}^{(\alpha)}(t).
\end{equation*}
The idea is now to make $\tau_i$ dependent on $\eta_{i-1}^{(\alpha)}$ instead of $\eta_{i-1}^{(0)}$.

\begin{definition}[Modified Modulo Hysteresis]
Let $\lambda > 0$, $h \geq 0$ and $\alpha \geq 0$.
For a function $g: \R \to \R$ and starting point $\tau_0 \in \R$, we define the two sequences $(\kappa_n)_{n \in \N_0}$ and $(\zeta_n)_{n \in \N_0}$ via
\begin{equation*}
\kappa_{n+1} = \inf\bigl\{t>\kappa_n \bigm| |\zeta_{n}(t)| \geq \lambda\bigr\}
\end{equation*}
and
\begin{equation*}
\zeta_{n+1} = \zeta_{n} - (2 \lambda - h) \sgn(\zeta_n(\kappa_{n+1})) \, \varepsilon_\alpha(\cdot - \kappa_{n+1}),
\end{equation*}
where we set $\kappa_0 = \tau_0$ and $\zeta_0 = g$.
With this, we define the {\em modified modulo hysteresis operator} $\Mod_\lambda^{h,\alpha}$ as
\begin{equation*}
\Mod_\lambda^{h,\alpha} g(t) = \lim_{n \to \infty} \zeta_n(t)
\quad \text{for } t \in \R.
\end{equation*}
\end{definition}

When considering the input function space
\begin{equation*}
\Cont_{\lambda,\tau_0}^{0,1} = \bigl\{g \in \Cont^{0,1}(\R) \bigm| |g(t)| < \lambda ~ \forall \, |t| \geq |\tau_0|\bigr\},
\end{equation*}
where $\Cont^{0,1}(\R)$ denotes the space of Lipschitz continuous functions on $\R$, which are differentiable almost everywhere with essentially bounded derivative, one can show that $\Mod_\lambda^{h,\alpha} g$ is well-defined for all $g \in \Cont_{\lambda,\tau_0}^{0,1}$ and, as opposed to $\Mod_H g$, it satisfies $\Mod_\lambda^{h,\alpha} g(t) \in [-\lambda,\lambda]$ for all $t \in \R$ and
\begin{equation*}
\kappa_0 < \kappa_n < \infty
\implies
|\Mod_\lambda^{h,\alpha} g(\kappa_n)| = \begin{cases}
\lambda & \text{if } \alpha > 0,\\
\lambda - h & \text{if } \alpha = 0,
\end{cases}
\end{equation*}
see also Fig.~\ref{fig:GenMod_example} for illustration.
Due to space limitations, we omit the proof in this paper and instead now focus on the identifiability from discrete modulo hysteresis samples
$
\{\Mod_\lambda^{h,\alpha}g(k\T) \mid k \in \Z\}
$
with sampling rate $\T > 0$.
To this end, we consider the space
\begin{equation*}
\Cont^{1,1}_{\lambda,\tau_0} = \bigl\{g \in \Cont^{0,1}_{\lambda,\tau_0} \bigm| g^\prime \in \Cont^{0,1}(\R)\bigr\}
\end{equation*}
and state a separation property of the folding points $(\kappa_n)_{n \in \N_0}$. 

\begin{lemma}
\label{lem:separation}
Let $\alpha > 0$ and $g \in \Cont^{1,1}_{\lambda,\tau_0}$ with $\|g^{\prime\prime}\|_\infty \leq \frac{2 h}{\alpha^2}$.
Then, for $n \in \N_0$ we have
\begin{equation*}
\Mod_\lambda^{h,\alpha} g(\kappa_n) \cdot \Mod^{h,\alpha}_\lambda g(\kappa_{n+1}) < 0
\implies
\kappa_{n+1} - \kappa_n \geq \alpha.
\end{equation*}
\end{lemma}

\begin{proof}
Consider the subsequence $ (\kappa_{m_n})_{n \in \N}$, where $\kappa_{m_1} = \kappa_{1}$ and $\kappa_{m_{n+1}}$ is the next folding point with different sign, i.e.,
\begin{equation*}
\kappa_{m_{n+1}} = \min_{k > m_n} \bigl\{\kappa_k \bigm| \sgn(\zeta_k(\kappa_k)) \neq \sgn(\zeta_{m_n}(\kappa_{m_n}))\bigr\} 
\end{equation*}
so that $\Mod_\lambda^{h,\alpha} g(\kappa_{m_n}) = \pm \lambda = -\Mod_\lambda^{h,\alpha} g(\kappa_{m_{n+1}})$.
Now, without loss of generality let $\Mod_\lambda^{h,\alpha} g(\kappa_{m_1}) = \lambda$.
Then, we obtain $\Mod_\lambda^{h,\alpha} g(\kappa_{m_2-1}) = \lambda$ and $\Mod_\lambda^{h,\alpha} g(\kappa_{m_2}) = -\lambda$.
Additionally, $g^\prime(\kappa_{m_2-1}) \geq p \, \frac{2 \lambda - h}{\alpha}$ with $p = |\{\kappa_k \in (\kappa_{m_2-1} - \alpha, \kappa_{m_2-1}  )\}|$.
Therefore, we can estimate $\zeta_{m_2-1}(t)$ for $t > \kappa_{m_2-1}$ as
\begin{align*}
&\zeta_{m_2-1}(t) \geq \zeta_{m_2-1}(\kappa_{m_2-1}) + g^\prime(\kappa_{m_2-1}) \,  (t - \kappa_{m_2-1}) \\ &\qquad - \frac{\|g^{\prime\prime}\|_\infty}{2} \, (t - \kappa_{m_2 -1})^2 - (p+1) \, \frac{2 \lambda - h}{\alpha} \, (t - \kappa_{m_2-1}) \\
&\quad\geq  \lambda  - \frac{h}{\alpha^2} \, (t - \kappa_{m_2-1})^2 - \frac{2 \lambda - h}{\alpha} \, (t - \kappa_{m_2-1}).
\end{align*}
With this we can conclude that
\begin{align*}
\kappa_{m_2} &\geq \inf\Bigl\{t > \kappa_{m_2-1} \Bigm| \lambda - \frac{h}{\alpha^2} \, (t - \kappa_{m_2-1})^2 \\ & \qquad \qquad - \frac{2 \lambda - h}{\alpha} \, (t - \kappa_{m_2-1}) = -\lambda\Bigr\} \\
&= \kappa_{m_2-1} + \alpha.
\end{align*}
We now assume that we have $\kappa_{m_n} \geq \kappa_{m_n-1} + \alpha$ for some $n \geq 2$.
Again without loss of generality let $\Mod_\lambda^{h,\alpha} g(\kappa_{m_n}) = \lambda$.
Since $\kappa_{m_{n+1}-1} \geq \kappa_{m_n} > \kappa_{m_n-1} + \alpha $ we can proceed with the same arguments as before to obtain
\begin{align*}
\kappa_{m_{n+1}} &\geq \inf\Bigl\{t>\kappa_{m_{n+1}-1} \Bigm| \lambda - \frac{h}{\alpha^2} \, (t - \kappa_{m_{n+1}-1})^2 \\ & \qquad \qquad  - \frac{2 \lambda - h}{\alpha} \, (t - \kappa_{m_{n+1}-1}) = -  \lambda\} \\
&= \kappa_{m_{n+1}-1} + \alpha,
\end{align*}
which completes the proof.
\end{proof}

\begin{algorithm}[t]
\setstretch{1.1}
\caption{(OMP)}
\label{alg:omp}
\begin{algorithmic}[1]
\Require Signal $\bfs \in \C^M$ and dictionary matrix $\bfV \in \C^{M \times N}$, error tolerance $\varepsilon > 0$
\medskip
\State $\bfc^{(0)} = 0 \in \C^N$, $\S^{(0)} = \emptyset$, $i = 1$
\While{$\lVert\bfV^\ast(\bfs - \bfV \bfc^{(i-1)})\rVert_\infty > \varepsilon $}
\smallskip
\State $j^{(i)} = \argmax_{1 \leq j \leq N} \lvert [\bfV^\ast(\bfs - \bfV \bfc^{(i-1)})]_j \rvert$
\State $\S^{(i)} = \S^{(i-1)} \cup \{j^{(i)}\}$
\State $\bfc^{(i)} = \argmin_{\bfc} \left\{\lVert\bfs - \bfV \bfc\rVert_2 \mid \supp(\bfc) \subseteq \S^{(i)}\right\}$
\State $i = i+1$
\smallskip
\EndWhile
\medskip
\Ensure $\bfc^{(i_{\text{end}})} \in \C^N$
\end{algorithmic}
\end{algorithm}

With the above condition on the second derivative on $g$ we can characterize the behaviour of $\Mod_\lambda^{h,\alpha} g(t)$ for large $t > |\tau_0|$.

\begin{lemma}
\label{lem:return}
If $g \in \Cont^{1,1}_{\lambda-h,\tau_0}$ with $\|g^{\prime\prime}\|_\infty \leq \frac{2 h}{\alpha^2}$ and $0 \leq h < \lambda$, we have $\Mod_\lambda^{h,\alpha} g(t) = g(t)$ for all $t \geq |\tau_0| + 2 \alpha$.
\end{lemma}

\begin{proof}
Let $\kappa_{n}$ be the largest folding point smaller than or equal to $|\tau_0|$.
We now go through the different possible cases.

\underline{Case 1:}
If $\kappa_n + \alpha \leq |\tau_0|$, there exists $k \in \Z$ such that
\begin{align*}
|\Mod_\lambda^{h,\alpha} g(|\tau_0|)| &= |k \, (2 \lambda - h) + g(|\tau_0|)| \\
&\geq |k| \, (2 \lambda - h) - |g(|\tau_0|)| \\
&> |k| \, (2 \lambda - h) - \lambda + h.
\end{align*}
As $|\Mod_\lambda^{h,\alpha} g(|\tau_0|)| \leq \lambda$, we get $k = 0$, which implies that $\Mod_\lambda^{h,\alpha} g(|\tau_0|) = g(|\tau_0|)$ and, thus, $\zeta_n(t) = g(t)$ for all $t \geq |\tau_0|$.
Inserting this into the definition of $\kappa_{n+1}$ gives	
\begin{equation*}
\kappa_{n+1} = \inf\{t > |\tau_0| \mid |g(t)| \geq \lambda\} = \infty,
\end{equation*}
from which follows that $\Mod_\lambda^{h,\alpha} g(t) = g(t) $ for all $t \geq |\tau_0|$.
As $\kappa_n \leq |\tau_0|$, this is the only case for $\alpha = 0$.
So now let $\alpha > 0$.

\underline{Case 2:}
If $\kappa_n + \alpha > |\tau_0|$, but $\kappa_n + \alpha \leq  \kappa_{n+1}$, we get, for some $k \in \Z$,
\begin{align*}
|\Mod_\lambda^{h,\alpha} g(\kappa_n + \alpha)| &= |k \, (2 \lambda - h) + g(\kappa_n + \alpha)| \\
&> |k| \, (2 \lambda - h) - \lambda + h
\end{align*}
so that $\kappa_{n+1} = \infty$ and $\Mod_\lambda^{h,\alpha} g(t) = g(t)$ for all $t \geq \kappa_n + \alpha$.
Since $\kappa_n \leq |\tau_0|$, this also holds for all $t \geq |\tau_0| + \alpha$.

\underline{Case 3:}
If $\kappa_n + \alpha > \max\{|\tau_0|, \kappa_{n+1}\}$, but $\kappa_{n+1} + \alpha \leq  \kappa_{n+2}$, we get, for some $k \in \Z$,
\begin{align*}
|\Mod_\lambda^{h,\alpha} g(\kappa_{n+1} + \alpha)| &= |k \, (2 \lambda - h) + g(\kappa_{n+1} + \alpha)| \\
&> |k| \, (2 \lambda - h) - \lambda + h
\end{align*}
so that $\kappa_{n+2} = \infty$ and $\Mod_\lambda^{h,\alpha} g(t) = g(t) $ for all $t \geq \kappa_{n+1} + \alpha$.
Since $\kappa_{n+1} \leq |\tau_0| + \alpha$, this also holds for all $t \geq |\tau_0| + 2\alpha$.

\underline{Case 4:}
Let $\kappa_n + \alpha > \max\{|\tau_0|, \kappa_{n+1}\}$, $\kappa_{n+1} + \alpha >  \kappa_{n+2}$.
As above, without loss of generality let $\Mod_\lambda^{h,\alpha} g( \kappa_{n}) = \lambda$.
Lemma \ref{lem:separation} implies that $\lambda = \Mod_\lambda^{h,\alpha} g(\kappa_{n+1}) = \Mod_\lambda^{h,\alpha} g(\kappa_{n+2})$ and $\Mod_\lambda^{h,\alpha} g(\kappa_{n+2} + \alpha) \leq \zeta_{n+2}(\kappa_{n+2} + \alpha)$.
This gives us that
\begin{align*}
& \Mod_\lambda^{h,\alpha} g( \kappa_{n+2} + \alpha) \leq \lambda + \zeta_{n+2}( \kappa_{n+2} + \alpha)  - \zeta_{n+2}( \kappa_{n} ) \\
&\quad \leq \lambda + g( \kappa_{n+2} + \alpha)  - g( \kappa_{n} ) - 3 ( 2 \lambda-h) \\
&\quad < \lambda + 2 (\lambda - h)- 3 (2 \lambda - h) = -3\lambda + h < -2\lambda
\end{align*}
in contradiction to $ \Mod_\lambda^{h,\alpha} g(t) \in [-\lambda,\lambda]$.
\end{proof}

We can now derive a condition for the sampling rate $\T$ to uniquely determine $g \in \PW_\Omega$ with bandwidth $\Omega > 0$ from modulo hysteresis samples.
In \cite[Lemma 1]{Bhandari2020}, it is shown that any function $g \in \PW_\Omega$ is uniquely determined by the samples $\{g(k\T_\varepsilon) \mid k \in \Z \setminus J\}$ if $ 0 < \T_{\varepsilon} \leq \frac{\pi}{\Omega+ \varepsilon}$ for arbitrarily small $\varepsilon > 0$ and any finite set $J \subset \Z$.
We can apply this to show that if $\T < \frac{\pi}{\Omega}$, $g$ is also uniquely determined by the samples $\{\A g(k\T) \mid k \in \Z\}$ for a general operator $\A: \PW_\Omega \to \Lebesgue^2(\R)$ with
\begin{equation*}
\A g(t) = g(t)
\quad \forall \, t \in \R \setminus M
\end{equation*}
 for some bounded set $M \subset \R$. 

\begin{lemma}
Any $g \in \PW_\Omega$ is determined by $\{\A g(k\T) \mid k \in \Z\}$ if the oversampling condition $\T < \frac{\pi}{\Omega}$ is met.
\end{lemma}

\begin{proof}
Since $\T < \frac{\pi}{\Omega}$, there exists $\varepsilon > 0$ so that $\T \leq \frac{\pi}{\Omega + \varepsilon}$.
Now define $J = M \cap \Z$, which is finite as $M \subset \R$ is bounded.
Since $\{\A g(k\T) \mid k \in \Z \setminus J\} = \{g(k\T) \mid k \in \Z \setminus J\}$ by assumption on $\A$, the statement follows from \cite[Lemma 1]{Bhandari2020}.
\end{proof}

Since the modulo hysteresis operator $\Mod_\lambda^{h,\alpha}$ satisfies the above conditions on $\A$ if we restrict ourselves to functions $g \in \PW_\Omega \cap \Cont^{1,1}_{\lambda-h,\tau_0}$ with $\|g^{\prime\prime}\|_\infty \leq \frac{2 h}{\alpha^2}$ and $0 \leq h < \lambda$ due to Lemma~\ref{lem:return}, we get the following identifiability result for $\Mod_\lambda^{h,\alpha}$.

\begin{corollary}
Let $0 \leq h < \lambda$.
Then, any $g \in \PW_\Omega \cap \Cont^{1,1}_{\lambda-h,\tau_0}$ with $\|g^{\prime\prime}\|_\infty \leq \frac{2 h}{\alpha^2}$ is uniquely determined by the modulo hysteresis samples $\{\Mod_\lambda^{h,\alpha} g(k\T) \mid k \in \Z\}$ if $\T < \frac{\pi}{\Omega}$.\qed
\end{corollary}

\begin{algorithm}[t]
\setstretch{1.1}
\caption{(SAOMP)}
\label{alg:saomp}
\begin{algorithmic}[1]
\Require Signal $\bfs \in \C^M$ and dictionary matrix $\bfV \in \C^{M \times N}$, error tolerance $\varepsilon > 0$, initial threshold $\nu \in [0,1]$, pruning threshold $\mu \in [0,1]$, maximal iteration number $i_{\text{max}} \in \N $
\medskip
\State $\bfc^{(0)} = 0 \in \C^N$, $\S^{(0)} = \emptyset$, $\delta = \nu$, $i = 1$
\While{$i \leq i_{\text{max}}$ \textbf{and} $\lVert\bfV^\ast(\bfs - \bfV \bfc^{(i-1)})\rVert_\infty > \varepsilon $}
\smallskip
\State $\bfr^{(i)} = \bfV^\ast(\bfs - \bfV \bfc^{(i-1)})$
\State $\S^{(i)} = \S^{(i-1)} \cup \{j \in \{1,\ldots,N\} \mid \lvert \bfr^{(i)}_j \rvert \geq \delta \, \lVert\bfr^{(i)}\rVert_\infty\}$
\State $\bfc^{(i)} = \argmin_{\bfc} \{\lVert\bfs - \bfV \bfc\rVert_2 \mid \supp(\bfc) \subseteq \S^{(i)}\}$
\State $\S^{(i)} = \{j \in \{1,\ldots,N\} \mid \lvert \bfc_j^{(i)} \rvert \geq \mu \, \lVert\bfc^{(i)}\rVert_\infty\}$
\State \textbf{for} $j \not\in \S^{(i)}$: $\bfc_j^{(i)} = 0$
\State $\delta = \delta + \frac{1-\nu}{i_{\text{max}}}$, $i = i+1$
\smallskip
\EndWhile
\medskip
\Ensure $\bfc^{(i_{\text{end}})} \in \C^N$
\end{algorithmic}
\end{algorithm}

\section{Reconstruction via OMP}\label{sec:OMP}

For our reconstruction approach we assume that we are given $2K + 1$ modulo hysteresis samples of $g \in \PW_\Omega \cap \Cont^{0,1}_{\lambda,\tau_0}$,
\begin{equation*}
\{\Mod_\lambda^{h,\alpha}g(k\T) \mid k=-K,\ldots,K\}
\end{equation*}
with sampling rate $0<\T<\frac{\pi}{\Omega}$, where $K \in \N$ is large enough such that $|g(t)| < \lambda$ for all $|t| \geq K\T$.
Let us first consider the folding function $\varepsilon_\alpha(\cdot - \kappa)$ located at fixed $-K\T < \kappa < K\T$ and set $\varepsilon_\alpha^{(\kappa)}[n] = \varepsilon_\alpha((n-K)\T - \kappa)$ for $n = 0,\ldots,N$ with $N = 2K$.
Let ${\Delta: \R^{N+1} \to \R^N}$ denote the forward difference operator defined by ${\Delta z[k] = z[k+1]-z[k]}$.
Then, there are at most $L = \lceil\frac{\alpha}{\T}\rceil + 1$ entries, where $\Delta \varepsilon_\alpha^{(\kappa)}$ is nonzero, so that $\Delta \varepsilon_\alpha^{(\kappa)}[n] = \Delta  \varepsilon_0^{(\tilde{\kappa}_1)}[n] + \ldots + \Delta \varepsilon_0^{(\tilde{\kappa}_L)}[n]$ for suitable $\tilde{\kappa}_1, \ldots, \tilde{\kappa}_L \in \T\Z$.
Intuitively, one linear fold is split into at most $L$ instantaneous folds.
Based on this observation we can adapt the recovery approach proposed in~\cite{Beckmann2022a}.

\begin{figure}[t]
\centering
\includegraphics[width=\linewidth]{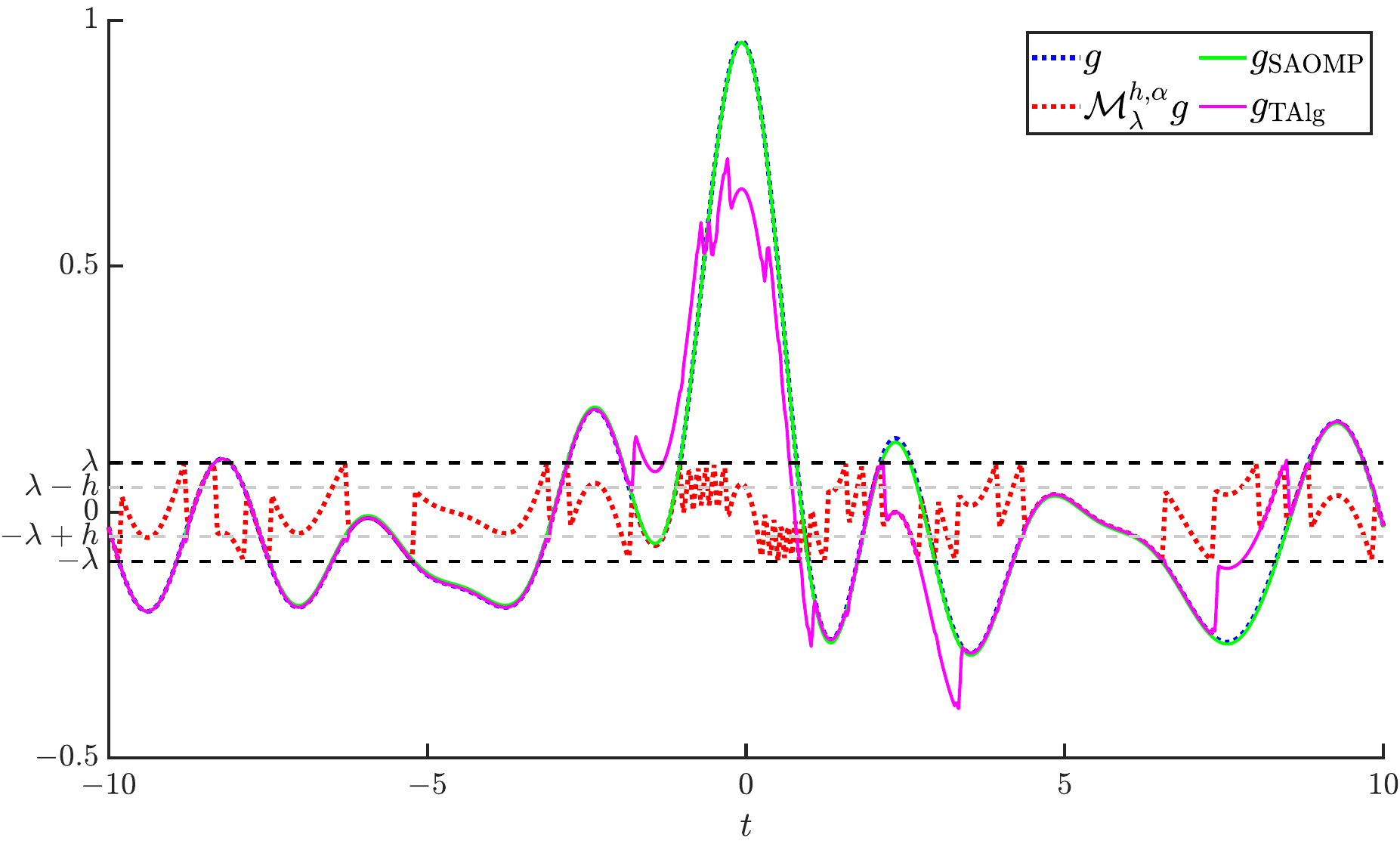}
\caption{Illustration of SAOMP and TAlg reconstruction for random $g \in \PW_\Omega$ with $\Omega = 6.3 \frac{\text{rad}}{\text{s}}$ based on modulo hysteresis samples $\Mod_\lambda^{h,\alpha} g(n\T)$ with $\lambda = 0.1$, $h = 0.05$, $\alpha =  50 \text{ms}$ and $\T = 20.8 \text{ms}$.}
\label{fig:SAOMP_example}
\end{figure}

To this end, set $g[n] = g((n-K)\T)$, for $n = 0,\ldots,N$, and $g_\lambda[n] = \Mod_\lambda^{h,\alpha} g((n-K)\T)$.
Then, there exist $L_\lambda \in \N$, $\kappa_1,\ldots,\kappa_{L_\lambda} \in [-K,K]\T$ and $\sigma_1,\ldots,\sigma_{L_\lambda} \in \{\pm1\}$ such that $g_\lambda[n] = g[n] - s_\lambda[n]$ with $s_\lambda[n] = (2\lambda-h) \sum_{l=1}^{L_\lambda} \sigma_l \, \varepsilon_\alpha^{(\kappa_l)}[n]$.
We apply $\Delta$ to get $\underline{g}_\lambda[n] = \Delta g_\lambda[n]$ for $n = 0,\ldots,N-1$ and, thereon, calculate its discrete Fourier transform (DFT) $\hat{\underline{g}}_\lambda[m] = \sum_{n=0}^{N-1} \underline{g}_\lambda[n] \exp(-\frac{2\pi\i mn}{N})$ for $m = 0,\ldots,N-1$.
Based on the above observation, the signal $\widehat{\underline{s}}_\lambda$ can be written as
\begin{equation*}
\widehat{\underline{s}}_\lambda[m] = \sum\nolimits_{\ell \in \Ell_\lambda} c_\ell \exp\Bigl(-\i \frac{\underline{\omega}_0 m}{\T}t_\ell\Bigr)
\end{equation*}
with $\underline{\omega}_0 = \frac{2\pi}{N}$, $\Ell_\lambda \subseteq \{0,\dots,N\}$ and $t_\ell \in (\T\Z) \cap [0,N\T]$, where $|\Ell_\lambda| \leq L_\lambda \, (\lceil\frac{\alpha}{\T}\rceil + 1)$.
Now, the bandlimitedness of $g$ gives
\begin{equation*}
\hat{\underline{g}}_\lambda[m] = -\hat{\underline{s}}_\lambda[m]
\quad \text{for } m \in \E_{N_\Omega,N}^\complement
\end{equation*}
with indices $\E_{N_\Omega,N}^\complement = \{N_\Omega+1,\ldots,N-N_\Omega-1\}$ and effective bandwidth $N_\Omega = \lceil\frac{\Omega(N+1)\T}{2\pi}\rceil$.
Defining the vector $\bfs \in \C^M$ for $M = N - 2 N_\Omega - 1$ via $\bfs_{m-N_\Omega} =  \hat{\underline{g}}_\lambda[m]$ for $m \in \E_{N_\Omega,N}^\complement$, we can find $\bfc = (c_n)_{n = 0}^{N-1} \in \C^{N}$ with $c_n = 0$ for $n \not\in \Ell_\lambda$ by solving the minimization problem 
\begin{equation}\label{eq:MinProb}
\text{minimize } \|\bfc\|_0
\quad \text{subject to}
\quad \bfV \bfc = \bfs,
\end{equation}
where $\bfV \in \C^{M \times N}$ is a Vandermonde matrix with entries $\bfV_{m-N_\Omega,n+1} = \e^{-\i \underline{\omega}_0 mn}$ for $m \in \E_{N_\Omega,N}^\complement$, $n = 0,\ldots,N-1$.
As explained in~\cite{Beckmann2022a}, \eqref{eq:MinProb} can be efficiently solved via the orthogonal matching pursuit (OMP) algorithm~\cite{Mallat1993}, see Algorithm~\ref{alg:omp}.
Applying anti-difference operator $\Sum: \R^N \to \R^{N+1}$, defined by $\Sum z[k] = \sum_{j < k} z[j]$, and subtracting $\Sum\bfc$ from $\bfg_\lambda = (g_\lambda[n])_{n = 0}^N$ finally recovers the samples $\bfg = (g[n])_{n = 0}^N$.

In~\cite{Beckmann2024}, a recovery guarantee for solving \eqref{eq:MinProb} using OMP is shown if the index of the maximal nonzero entry of $\bfc \in \C^N$ satisfies $\max \{n \mid \bfc_n \neq 0\} \leq M-1$.
By means of Lemma~\ref{lem:return} we know that this is satisfied if $g \in \Cont^{1,1}_{\lambda-h,\tau_0}$ with $\|g^{\prime\prime}\|_\infty \leq \frac{2 h}{\alpha^2}$ and
\begin{equation*}
\biggl\lceil \frac{|\tau_0|+ 2 \alpha}{\T}\biggr\rceil + K \leq N -2 N_{\Omega} - 2.	
\end{equation*}

To accelerate the recovery, we replace OMP by the so-called stagewise arithmetic orthogonal matching pursuit (SAOMP) algorithm proposed in~\cite{Zhang2018}.
It has additional parameters $\nu$ to find multiple folding points in one iteration and $\mu$ to remove incorrectly selected ones, see Algorithm~\ref{alg:saomp}.
Note that SAOMP agrees with OMP when choosing $\nu = 1$, $\mu = 0$ and $i_{\text{max}} = N$.

\section{Numerical Experiments}\label{sec:numerics}
	
In our numerical experiments we randomly select a function $g \in \PW_\Omega$ with $\Omega = 6.3$ and approximate its modulo hysteresis output $\Mod_\lambda^{h,\alpha}g$ numerically.
Thereon, we reconstruct $g$ from its samples $\Mod_\lambda^{h,\alpha}g(n\T)$ using our SAOMP approach and compare our results with the thresholding algorithm (TAlg) proposed in~\cite{Florescu2022}.
One example is shown in Fig.~\ref{fig:SAOMP_example}, where we choose $\lambda = 0.1$, $h = 0.05$, $\alpha = 0.05$ and $\T = 0.0208$.

While recovery with TAlg performs better for $\alpha \ll 1$ very close to zero, TAlg fails for larger $\alpha$, whereas SAOMP still recovers, see Fig.~\ref{fig:SAOMP_example}.
For a more detailed analysis, we investigate the influence of $\alpha$ and $h$ on the reconstruction success of our SAOMP approach.
To this end, we set $\lambda = 0.1$ and vary $\alpha$ between $0$ and $0.07$ and $h$ between $0$ and $0.1$.
We then calculate the mean squared error (MSE) for the SAOMP reconstructions of $50$ randomly chosen functions.
The results are depicted in Fig.~\ref{fig:SAOMP_alpha_h}, where the colors indicate for how many functions the MSE of SAOMP is larger than $0.001$.

We observe that SAOMP gives satisfactory reconstructions for a large range of parameters.
However, our approach shows instabilities for a large number of folds, in particular if $\alpha$ is much larger than $\T$.
In this case, a modified dictionary matrix $\bfV$ can reduce the number of non-zero elements in the coefficient vector $\bfc$ and allows for better reconstruction results.
This approach, however, is beyond the scope of this paper and calls for more in-depth research in the future.

\begin{figure}[t]
\centering
\includegraphics[height=4.5cm]{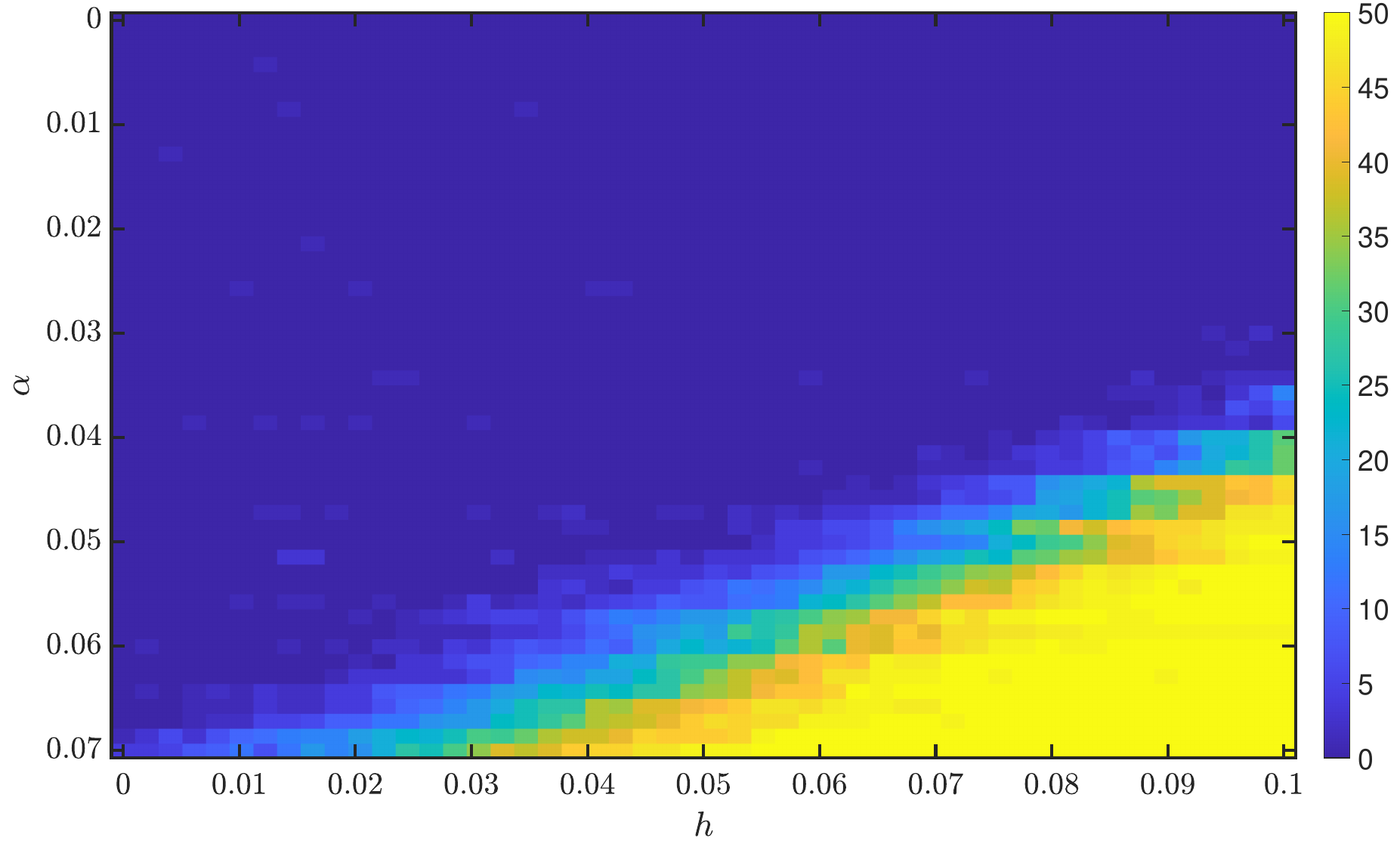}
\caption{Success of SAOMP reconstruction for $\Mod_\lambda^{h,\alpha}g(n\T)$ with $\lambda = 0.1$, $\T = 20.8 \text{ms}$ and different values for $\alpha$ and $h$ showing how often the MSE is larger than $0.001$ for $50$ random functions $g \in \PW_\Omega$ with $\Omega = 6.3 \frac{\text{rad}}{\text{s}}$.}
\label{fig:SAOMP_alpha_h}
\end{figure}


\vfill\pagebreak

\clearpage
\balance

\end{document}